\documentclass[11pt]{article}
\usepackage{amsmath,amssymb,graphicx,bbm,bm,delarray,pict2e,ntheorem}
\usepackage[hypertex]{hyperref}
\usepackage{graphicx}
\usepackage[usenames]{color}
\usepackage[normalem]{ulem}
\long\def\ignore#1{}

\begin{document}

\newtheorem{theorem}{Theorem}
\newtheorem{prop}[theorem]{Proposition}
\newtheorem{lemma}[theorem]{Lemma}
\newtheorem{claim}{Claim}
\newtheorem{corollary}[theorem]{Corollary}
\theorembodyfont{\rmfamily}
\newtheorem{remark}[theorem]{Remark}
\newtheorem{conjecture}[theorem]{Conjecture}
\newtheorem{problem}[theorem]{Problem}
\newtheorem{step}{Step}
\newtheorem{alg}{Algorithm}
\newenvironment{proof}{\medskip\noindent{\bf Proof. }}{\hfill$\square$\medskip}

\def\R{\mathbb{R}}
\def\one{\mathbbm1}
\def\T{^{\sf T}}
\def\HH{{\cal H}}
\def\SS{{\cal S}}
\def\II{{\cal I}}
\def\FF{{\cal F}}
\def\AA{{\cal A}}
\def\EE{{\cal E}}
\def\LL{{\cal L}}
\def\PP{{\cal P}}
\def\CC{{\cal C}}
\def\MM{{\cal M}}
\def\XX{{\cal X}}
\def\WW{{\cal W}}
\def\P{{\sf P}}
\def\E{{\sf E}}
\def\Q{{\mathbf Q}}
\def\bd{\text{bd}}
\def\eps{\varepsilon}

\newcommand{\NIET}[1]{}
\newcommand{\bx}{\hspace*{\fill} \hbox{\hskip 1pt \vrule width 4pt height 8pt depth 1.5pt \hskip 1pt}

\addvspace{4mm}}
\newcommand{\openbx}{\hspace*{\fill} $\Box$\\ \vspace{1mm}}

\def\wh{\widehat}
\def\cork{\text{\rm corank}}
\def\rank{\text{\rm rank}}
\def\Ker{\text{\rm Ker}}
\def\rk{\text{\rm rank}}
\def\supp{\text{\rm supp}}
\def\diag{\text{\rm diag}}

\begin{center}

{\Large\bf Nullspace embeddings for outerplanar graphs}\\[12mm]
L\'{a}szl\'o Lov\'{a}sz\footnote{
The research leading to these results has received funding from the European Research Council
under the European Union's Seventh Framework Programme (FP7/2007-2013) / ERC grant agreement
n$\mbox{}^{\circ}$ 227701.
}\\
E\"otv\"os Lor\'and University, Budapest, Hungary
\\
and \\
Alexander Schrijver\footnote{
The research leading to these results has received funding from the European Research Council
under the European Union's Seventh Framework Programme (FP7/2007-2013) / ERC grant agreement
n$\mbox{}^{\circ}$ 339109.
}\\
University of Amsterdam and CWI, Amsterdam, The Netherlands\\[1cm]

{\it Dedicated to the memory of Ji\v{r}\'{\i} Matou\v{s}ek}
\end{center}

\tableofcontents

\begin{abstract}
We study relations between geometric embeddings of graphs and the
spectrum of associated matrices, focusing on outerplanar embeddings
of graphs. For a simple connected graph $G=(V,E)$, we define a
``good'' $G$-matrix as a $V\times V$ matrix with negative entries
corresponding to adjacent nodes, zero entries corresponding to
distinct nonadjacent nodes, and exactly one negative eigenvalue. We
give an algorithmic proof of the fact that it $G$ is a 2-connected
graph, then either the nullspace representation defined by any
``good'' $G$-matrix with corank 2 is an outerplanar embedding of $G$,
or else there exists a ``good'' $G$-matrix with corank 3.
\end{abstract}

\section{Introduction}

We study relations between geometric embeddings of graphs, the
spectrum of associated matrices and their signature, and topological
properties of associated cell complexes. We focus in particular on
$1$-dimensional and $2$-dimensional embeddings of graphs, in the hope
that the techniques can be extended to higher dimensions.

\medskip

{\bf Spectral parameters of graphs.} The basic connection between
graphs, matrices, and geometric embeddings considered in this paper
can be described as follows. We define a {\it $G$-matrix} for an
undirected graph $G=(V,E)$ as a symmetric real-valued $V\times V$
matrix that has a zero in position $(i,j)$ if $i$ and $j$ are
distinct nonadjacent nodes. The matrix is {\it well-signed} if
$M_{ij}<0$ if $i$ and $j$ are distinct adjacent nodes. (There is no
condition on the diagonal entries.)  If, in addition, $M$ has exactly
one negative eigenvalue, then let us call it {\it good} (for the
purposes of this introduction). Let $\kappa(G)$ denote the largest
$d$ for which there exists a good $G$-matrix.

The parameter $\kappa$ is closely tied to certain topological
properties of the graph. Combining results of \cite{CdV}, \cite{HH},
\cite{RST}, \cite{LS1} and \cite{SS}, one gets the following facts:
\begin{quote}
If $G$ is connected, then $\kappa(G)\leq 1\Leftrightarrow G$ is a path,\\
If $G$ is $2$-connected, then $\kappa(G)\leq 2\Leftrightarrow G$ is outerplanar,\\
If $G$ is $3$-connected, then $\kappa(G)\leq 3\Leftrightarrow G$ is planar,\\
If $G$ is $4$-connected, then $\kappa(G)\leq 4\Leftrightarrow G$ is
linklessly embeddable.
\end{quote}

We study algorithmic aspects of the first two facts. Let us discuss
here the second, which says that if $G$ is a $2$-connected graph,
then either it has an embedding in the plane as an outerplanar map,
or else there exists a good $G$-matrix with corank 3 (and so the
graph is not outerplanar). To construct an outerplanar embedding, we
use the nullspace of any good $G$-matrix with corank $2$.

\medskip

{\bf Nullspace representations.} To describe this construction,
suppose that a $V\times V$ matrix $M$ has corank $d$. Let
$U\in\R^{d\times n}$ be a matrix whose rows form a basis of the
nullspace of $M$. This matrix satisfies the equation
\begin{equation}\label{EQ:NULLSP}
UM=0,
\end{equation}
where $U$ is a $d\times n$ matrix of rank $d$ and $M$ is a $G$-matrix
with corank $d$. Let $u_i$ be the column of $U$ corresponding to node
$i\in V$. The mapping $u:~V\to\R^d$ is called the {\it nullspace
representation of $V$ defined by $M$}. It is unique up to linear
transformations of $\R^d$. (For the purist: the map $V\to\ker(M)^*$
is canonically defined; choosing the basis in $\ker(M)$ just
identifies $\ker(M)$ with $\R^d$.)

If $G=(V,E)$ is a graph and $u:~V\to\R^d$ is any map, we can extend
it to the edges by mapping the edge $ij$ to the straight line segment
between $u_i$ and $u_j$. If $u$ is the nullspace representation of
$V$ defined by $M$, then this extension gives the {\it nullspace
representation of} $G$ {\it defined by} $M$.

In this paper we give algorithmic proofs of two facts:
\begin{enumerate}
\item[(1)] If $G$ is connected graph with $\kappa(G)=1$, then the
    nullspace representation defined by any well-signed
    $G$-matrix $M$ with one negative eigenvalue and with corank
    $1$ yields an embedding of $G$ in the line.

\item[(2)] If $G$ is $2$-connected and $\kappa(G)=2$, then the
    nullspace representation defined by any well-signed
    $G$-matrix $M$ with one negative eigenvalue and with corank
    $2$ yields an outerplanar embedding of $G$.
\end{enumerate}
The proofs are algorithmic in the sense that (say, in the case of
(2)) for every $2$-connected graph we either construct an outerplanar
embedding or a good $G$-matrix with corank $3$ in polynomial time.
The alternative proof that can be derived from the results of
\cite{LS2} uses the minor-monotonicity of the Colin de Verdi\`ere
parameter (see below), and this way it involves repeated reference to
the Implicit Function Theorem, and does not seem to be implementable
in polynomial time. The word "yields" above hides some issues
concerning normalization, to be discussed later.

Paper \cite{LS2} also contains the analogous result for planar
graphs, which was extended in \cite{L}:
\begin{enumerate}
\item[(3)] If $G$ is $3$-connected and $\kappa(G)=3$, then the
    nullspace representation defined by any well-signed
    $G$-matrix with one negative eigenvalue and with corank $2$
    yields a representation of $G$ as the skeleton of a convex
    $3$-polytope.
\end{enumerate}
Again, the proof uses the minor-monotonicity of the Colin de
Verdi\`ere parameter and the Implicit Function Theorem, and thus it
is not algorithmic. It would be interesting to see whether our
approach can be extended to an algorithmic proof for dimension 3.
(While we focus on the case $\kappa(G)=2$, some of our results do
bear upon higher dimensions, in particular the results in Section
\ref{26ok15a} below.)

A further extension to dimension $4$ would be particularly
interesting, since $4$-connected linklessly embeddable graphs are
characterized by the property that $\kappa(G)\le 4$, but it is not
known whether the nullspace representation obtained from a good
$G$-matrix of corank $4$ yields a linkless embedding of the graph.

\medskip

{\bf The Strong Arnold Hypothesis and the Colin de Verdi\`ere
number.} We conclude this introduction with a discussion of the
connection between the parameter $\kappa(G)$ and the graph parameter
$\mu(G)$ introduced by Colin de Verdi\`ere \cite{CdV}. This latter is
defined similarly to $\kappa$ as the maximum corank of a good
$G$-matrix $M$, where it is required, in addition, that $M$ has a
nondegeneracy property called the {\it Strong Arnold Property}. There
are several equivalent forms of this property; let us formulate one
that is related to our considerations in the sense that it uses any
nullspace representation $u$ defined by $M$: if a symmetric $d\times
d$ matrix $N$ satisfies $u_i\T Nu_i=0$ for all $i\in V$ and $u_i\T
Nu_j=0$ for each edge $ij$ of $G$, then $N=0$. In more geometric
terms this means that the nullspace representation of the graph
defined by $M$ is not contained in any nontrivial homogeneous
quadric.

The relationship between $\mu$ and $\kappa$ is not completely
clarified. Trivially $\mu(G)\leq\kappa(G)$. Equality does not hold in
general: consider the graph $G_{l,m}$ made from an $(l+m)$-clique by
removing the edges of an $m$-clique. If $l\geq 1$ and $m\geq 3$, then
$\mu(G_{l,m})=l+1$ whereas $\kappa(G_{l,m})=l+m-2$. (Note that
$G_{l,m}$ is not $l+1$-connected.)

Colin de Verdi\`ere's parameter has several advantages over $\kappa$.
First, it is minor-monotone, while $\kappa(G)$ is not minor-monotone,
not even subgraph-monotone: any path $P$ satisfies $\kappa(P)\leq 1$,
but a disjoint union of paths can have arbitrarily large $\kappa(G)$.
Furthermore, the connection with topological properties of graphs
holds for $\mu$ without connectivity conditions:
\begin{quote}
$\mu(G)\leq 1$ $\Leftrightarrow$ $G$ is a disjoint union of paths,\\
$\mu(G)\leq 2$ $\Leftrightarrow$ $G$ is outerplanar,\\
$\mu(G)\leq 3$ $\Leftrightarrow$ $G$ is planar,\\
$\mu(G)\leq 4$ $\Leftrightarrow$ $G$ is linklessly embeddable in
$\R^3$.
\end{quote}
Our use of $\kappa$ is motivated by its easier definition and by the
(slightly) stronger, algorithmic results.

We see from the facts above that by requiring that $G$ is
$\mu(G)$-connected, we have $\mu(G)=\kappa(G)$ for $\mu(G)\le 4$. In
fact, it was shown by Van der Holst \cite{HolstThesis} that if $G$ is
$2$-connected outerplanar or $3$-connected planar, then {\it every}
good $G$-matrix has the Strong Arnold Property. This also holds true
for 4-connected linklessly embeddable graphs \cite{SS}. One may
wonder whether this remains true for $\mu(G)$-connected graphs with
larger $\mu(G)$. This would imply that $\mu(G)=\kappa(G)$ for every
$\mu(G)$-connected graph.

\section{$G$-matrices}

\subsection{Nullspace representations}

Let us fix a connected graph $G=(V,E)$ on node set $V=[n]$, and an
integer $d\ge 1$. We denote by $\WW=\WW(G,d)$ the set of well-signed
$G$-matrices with corank at least $d$, and by $\WW_0=\WW_0(G,d)$, the
set of well-signed $G$-matrices with corank exactly $d$. We define
$\WW'=\WW'(G,d)$ as the set of $G$-matrices in $\WW(G,d)$ with
exactly one negative eigenvalue (of multiplicity 1). We denote by
$\MM_u$ the linear space of $G$-matrices $M$ with $UM=0$, by $\WW_u$,
the set of well-signed $G$-matrices in $\MM_u$, and by $\WW'_u$, the
set of matrices in $\WW_u$ with exactly one negative eigenvalue.

We can always perform a linear transformation of $\R^d$, i.e.,
replace $U$ by $AU$, where $A$ is any nonsingular $d\times d$ matrix.
In the case when $\cork(M)=d$ (which will be the important case for
us), the matrix $U$ is determined by $M$ up to such a linear
transformation of $\R^d$.

Another simple transformation we use is ``node scaling'': replacing
$U$ by $U'=UD$ and $M$ by $M'=D^{-1}MD^{-1}$, where $D$ is a
nonsingular diagonal matrix. Then $M'$ is a $G$-matrix and $U'M'=0$.
Through this transformation, we may assume that every nonzero vector
$u_i$ has unit length. We call such a representation {\em
normalized}.

One of our main tools will be to describe more explicit solutions of
the basic equation \eqref{EQ:NULLSP} in dimensions $1$ and $2$. More
precisely, given a graph $G=(V,E)$ and a representation
$U:~V\to\R^2$, our goal is to describe all $G$-matrices $M$ with
$UM=0$. Note that it suffices to find the off-diagonal entries: if
$M_{ij}$ is given for $ij\in E$ in such a way that
\begin{equation}\label{EQ:PAR}
\sum_{j\in N(i)} M_{ij}u_j ~\parallel~ u_i,
\end{equation}
then there is a unique choice of diagonal entries $M_{ii}$ that gives
a matrix with $UM=0$:
\begin{equation}\label{8ok15d}
M_{ii}=-\sum_j M_{ij}\frac{u_j\T u_i}{u_i\T u_i}.
\end{equation}

\subsection{$G$-matrices and eigenvalues\label{26ok15a}}

In this section we consider eigenvalues of well-signed $G$-matrices;
we consider the connected graph $G$ and the dimension parameter $d$
fixed. We start with a couple of simple observations.

\begin{lemma}\label{PROP:CONVHULL1}
Let $M$ be a well-signed $G$-matrix with corank $d\ge 1$ and let
$U\in\R^{d\times n}$ such that $UM=0$ and $\rank(U)=d$.

\smallskip

{\rm(a)} If $M$ is positive semidefinite, then $d=1$, and all entries
of $U$ are nonzero and have the same sign.

\smallskip

{\rm(b)} If $M$ has a negative eigenvalue, then the origin is an
interior point of the convex hull of the columns of $U$.
\end{lemma}

\begin{proof}
Let $\lambda$ be the smallest eigenvalue of $M$. As $G$ is connected,
$\lambda$ has multiplicity one by the Perron--Frobenius theorem, and
$M$ has a positive eigenvector $v$ belonging to $\lambda$. If
$\lambda=0$, then this multiplicity is $d=1$, and $U$ consists of a
single row parallel to $v$. If $\lambda<0$, then every row of $U$,
being in the nullspace of $M$, is orthogonal to $v$. Thus the entries
of $v$ provide a representation of $0$ as a convex combination of the
columns of $U$ with positive coefficients.
\end{proof}

\begin{lemma}\label{SHIFT-NEW}
If $d\geq 2$, then the set $\WW'$ is relatively closed in $\WW$, and
$\WW'\cap \WW_0$ is relatively open in $\WW$.
\end{lemma}

\begin{proof}
Let $\lambda_i(M)$ denote the $i$-th smallest eigenvalue of the
matrix $M$. We claim that for any $M\in\WW$,
\begin{equation}\label{9ok15a}
M\in\WW' ~\Leftrightarrow~\lambda_2(M)\geq 0.
\end{equation}
Indeed, if $M\in\WW'$, then trivially $\lambda_2(M)\geq 0$.
Conversely, if $\lambda_2(M)\geq 0$, then $M$ has at most one
negative eigenvalue. By Lemma \ref{PROP:CONVHULL1}(a), it has exactly
one, that is, $M\in\WW'$. This proves \eqref{9ok15a}. Since
$\lambda_2(M)$ is a continuous function of $M$, the first assertion
of the lemma follows.

We claim that if $d\ge 2$, for any $M\in\WW$,
\begin{equation}\label{9ok15ax}
M\in\WW'\cap\WW_0 ~\Leftrightarrow~\lambda_{d+2}(M)> 0.
\end{equation}
Indeed, if $M\in\WW'\cap\WW_0$, then $M$ has one negative eigenvalue
and exactly $d$ zero eigenvalues, and so $\lambda_{d+2}(M)>0$.
Conversely, assume that $\lambda_{d+2}(M)>0$. Since $M$ has at least
$d$ zero eigenvalues and at least one negative eigenvalue (by Lemma
\ref{PROP:CONVHULL1}(a)), we must have equality in both bounds, which
means that $M\in\WW'\cap\WW_0$. This proves \eqref{9ok15ax}.
Continuity of $\lambda_{d+2}(M)$ implies the second assertion.
\end{proof}

This last proposition implies that each nonempty connected subset of
$\WW_0$ is either contained in $\WW'$ or is disjoint from $\WW'$. We
formulate several consequences of this fact.

\begin{lemma}\label{27ok15e}
Suppose that $G$ is $2$-connected, and let $M$ be a well-signed
$G$-matrix with one negative eigenvalue and with corank
$d=\kappa(G)$. Let $u$ be the nullspace representation defined by
$M$, let $v\in\R^d$, and let $J:=\{i:~ u_i=v\}$. If $|J|\ge2$, then
the origin $0$ belongs to the convex hull of $u(V\setminus J)$.
\end{lemma}

\begin{proof}
For $i\in V$, let $e_i$ be the $i$-th unit basis vector, and for
$i,j\in  V$, let $D^{ij}$ be the matrix $(e_i-e_j)(e_i-e_j)\T$.
Define
\[
M^{\alpha}:=M+\alpha\sum_{ij\in E\atop i,j\in J}M_{ij}D^{ij}
\qquad(\alpha\in[0,1]).
\]
The definition of $J$ implies that $\ker(M)\subseteq \ker(D^{ij})$
for all $i,j\in J$, and hence $\ker(M)\subseteq \ker(M^{\alpha})$ for
each $\alpha\in[0,1]$. So $\cork(M^{\alpha})\geq\cork(M)=\kappa(G)$
for each $\alpha\in[0,1]$. Moreover, $M^{\alpha}$ is a well-signed
$G$-matrix for each $\alpha\in[0,1)$. Since $M=M^0\in\WW'$, Lemma
\ref{SHIFT-NEW} implies that $M^{\alpha}\in\WW'$ for each
$\alpha\in[0,1)$. By the continuity of eigenvalues, $M^1$ has at most
one negative eigenvalue. Note that $M^1_{ij}=0$ for any two distinct
$i,j\in J$.

Assume that $0$ does not belong to the convex hull of $\{u_i:~
i\not\in J\}$. Then there exists $c\in\R^{\kappa(G)}$ such that
$u_i\T c<0$ for each $i\not\in J$. As $0$ belongs to interior of the
convex hull of $u(V)$ by Lemma \ref{PROP:CONVHULL1}(b), this implies
that $u_i\T c= v\T c >0$ for each $i\in J$.

As $|J|\geq 2$, the 2-connectivity of $G$ implies that $J$ contains
two distinct nodes, say nodes 1 and 2, that have neighbors outside
$J$. Since $\ker(M)\subseteq\ker(M^1)$, we have $\sum_j M^1_{1j}u_j
=0$, and hence
\[
M^1_{11}u_1\T c=-\sum_{j\neq 1}M^1_{1j}u_j\T c=
-\sum_{j\not\in J}M^1_{1j}u_j\T c.
\]
As $u_1\T c>0$ and $u_j\T c<0$ for $j\not\in J$, and as $M^1_{1j}\leq
0$ for all $j\not\in J$, and $M^1_{1j}<0$ for at least one $j\not\in
J$, this implies $M^1_{11}<0$. Similarly, $M^1_{22}<0$. As
$M^1_{12}=0$, the first two rows and columns of $M'$ induce a
negative definite $2\times 2$ submatrix of $M^1$. This contradicts
the fact that $M^1$ has at most one negative eigenvalue.
\end{proof}

For the next step we need a simple lemma from linear algebra.

\begin{lemma}\label{ADD1}
Let $A$ and $M$ be symmetric $n\times n$ matrices. Assume that $A$ is
$0$ outside a $k \times k$ principal submatrix, and let $M_0$ be the
complementary $(n-k)\times(n-k)$ principal submatrix of $M$. Let $a$
and $b$ denote the number of negative eigenvalues of $A$ and $M_0$,
respectively. Then for some $s>0$, the matrix $sM+A$ has at least
$a+b$ negative eigenvalues.
\end{lemma}

\begin{proof}
We may assume $A=\scriptsize\begin{pmatrix}A_0&0\\0&0\end{pmatrix}$
and $M=\scriptsize\begin{pmatrix}M_1&M_2\T\\M_2&M_0\end{pmatrix}$,
with $A_0$ and $M_1$ having order $k\times k$. By scaling the last
$n-k$ rows and columns of $sM+A$ by $1/\sqrt{s}$, we get the matrix
$\scriptsize\begin{pmatrix} sM_1+A_0&\sqrt{s}M_2\T\\
\sqrt{s}M_2&M_0\end{pmatrix}$. Letting $s\to 0$, this tends to
$B=\scriptsize\begin{pmatrix} A_0&0\\
0&M_0\end{pmatrix}$. Clearly, $B$ has $a+b$ negative eigenvalues,
and by the continuity of eigenvalues, the lemma follows.
\end{proof}

\begin{lemma}\label{27ok15b}
Let $M$ be a well-signed $G$-matrix with one negative eigenvalue and
with corank $d=\kappa(G)$, let $u$ be the nullspace representation
defined by $M$, and let $C$ be a clique in $G$ of size at most
$\kappa(G)$ such that the origin belongs to the convex hull of
$u(C)$. Then $G-C$ is disconnected.
\end{lemma}

\begin{proof}
We can write $0=\sum_ia_iu_i$ with $a_i\geq 0$, $\sum_ia_i=1$, and
$a_i=0$ if $i\not\in C$. Let $A$ be the matrix $-aa\T$. Since $a$ is
nonzero, $A$ has a negative eigenvalue.

Since $\sum_ia_iu_i=0$, we have $\ker(M)\subseteq\ker(M+sA)$ for each
$s$. This implies that $\cork(M+sA)\geq\cork(M)$ for each $s$.
Moreover, $M+sA$ is a well-signed $G$-matrix for $s\geq 0$.
Hence, as $M\in\WW'$, we know by Lemma \ref{SHIFT-NEW} that
$M+sA\in\WW'$ for every $s\geq 0$. In other words, $M+sA$ has one
negative eigenvalue for every $s\geq 0$.

Let $M_0$ be the matrix obtained from $M$ by deleting the rows and
columns with index in $C$. Note that $M_0$ has no negative
eigenvalue: otherwise by Lemma \ref{ADD1}, $M+sA$ has at least two
negative eigenvalues for some $s>0$, a contradiction.

Now suppose that $G-C$ is connected. As $u(C)$ is linearly dependent
and $|C|\le\cork(M)$, $\ker(M)$ contains a nonzero vector $x$ with
$x_i=0$ for all $i\in C$. Then by the Perron--Frobenius theorem,
$\cork(M_0)=1$ and $\ker(M_0)$ is spanned by a positive vector $y$.
As $G$ is connected, $x$ is orthogonal to the positive eigenvector
belonging to the negative eigenvalue of $M$. So $x$ has both positive
and negative entries. On the other hand, $x|_{V\setminus
C}\in\ker(M_0)$, and so $x|_{V\setminus C}$ must be a multiple of
$y$, a contradiction.
\end{proof}

Taking $C$ a singleton, we derive:

\begin{corollary}\label{1NODE}
Let $G$ be a $2$-connected graph, let $M\in\WW'$ have corank
$\kappa(G)$, and let $u$ be the nullspace representation defined by
$M$. Then $u_i\neq 0$ for all $i$. Equivalently, the nullspace
representation defined by $M$ can be normalized by node scaling.
\end{corollary}

\section{1-dimensional nullspace representations}\label{SEC:1DIM}

As a warmup, let us settle the case $d=1$. For every connected
graph $G=(V,E)$, it is easy to construct a singular $G$-matrix with
exactly one negative eigenvalue: start with any $G$-matrix, and
subtract an appropriate constant from the main diagonal. Our goal is
to show that unless the graph is a path and the nullspace
representation is a monotone embedding in the line, we can modify the
matrix to get a $G$-matrix with one negative eigenvalue and with
corank at least $2$.

\subsection{Nullspace and neighborhoods}

We start with noticing that given vector $u\in\R^V$, it is easy to
describe the matrices in $\WW_u$. Indeed, consider any matrix
$M\in\MM_u$. Then for every node $i$ with $u_i=0$, we have
\begin{equation}\label{EQ:MSUM}
\sum_{j\in N(i)} M_{ij} u_j = \sum_j M_{ij}u_j= 0.
\end{equation}
Furthermore, for every node $i$ with $u_i\not=0$, we have
\begin{equation}\label{EQ:SUPPU}
M_{ii}= -\frac1{u_i}\sum_{j\in N(i)} M_{ij} u_j.
\end{equation}
Conversely, if we specify the off-diagonal entries of a $G$-matrix
$M$ so that \eqref{EQ:MSUM} is satisfied, then we can define $M_{ii}$
for nodes $i\in \supp(u)$ according to \eqref{EQ:SUPPU}, and for
nodes $i$ with with $u_i=0$ arbitrarily, we get a matrix in $\MM_u$.

As an application of this construction, we prove the following lemma.

\begin{lemma}\label{LEM:MU}
Let $u\in \R^V$. Then $\WW_u\not=\emptyset$ if and only if for every
node $i$ with $u_i=0$, either all its neighbors satisfy $u_j=0$, or
it has neighbors both with $u_j<0$ and $u_j>0$.
\end{lemma}

\begin{proof}
By the remark above, it suffices to specify negative numbers $M_{ij}$
for the edges $ij$ so that \eqref{EQ:MSUM} is satisfied for each $i$
with $u_i=0$. The edges between two nodes with $u_i=0$ play no role,
and so the conditions \eqref{EQ:MSUM} can be considered separately.
For a fixed $i$, the single linear equation for the $M_{ij}$ can be
satisfied by negative numbers if and only if the condition in the
lemma holds.
\end{proof}

We need the following fact about the neighbors of the other nodes.

\begin{lemma}\label{LEM:ORDER}
Let $u\in\R^V$, $M\in\WW_u$, and suppose that $M$ has a negative
eigenvalue $\lambda<0$, with eigenvector $\pi>0$. Then every node $i$
with $u_i>0$ has a neighbor $j$ for which $u_j/\pi_j<u_i/\pi_i$.
\end{lemma}

\begin{proof}
Suppose not. Then $u_j\ge \pi_j u_i/\pi_i$ for every $j\in N(i)$, and
so
\[
0=\sum_j M_{ij}u_j \le  M_{ii} u_i +
\sum_{j\in N(i)} M_{ij}\frac{\pi_j}{\pi_i}u_i
= \frac{u_i}{\pi_i}\Bigl(\sum_j M_{ij}\pi_j\Bigr)= \lambda u_i <0,
\]
a contradiction.
\end{proof}

\subsection{Auxiliary algorithms}\label{SEC:ALG-AUX}

No we turn to the algorithmic part, starting with some auxiliary
algorithms.

\begin{alg}[Interpolation]\label{ALG:INTER} ~

{\it Input:} a vector $u\in\R^V$ and two matrices $M\in\WW'_u$ and
$M'\in\WW_u\setminus\WW'_u$.

{\it Output:} a matrix $M''\in\WW'_u$ with corank at least $2$.
\end{alg}

Consider the well-signed $G$-matrices $M^t=tM'+(1-t)M$ $(0\le t\le
1)$. As $\WW'\cap\WW_0$ is open and closed in $\WW_0$, there must be
points $t\in[0,1]$ where $\cork(M^t)>1$. We can find these values $t$
by considering any nonsingular $(n-1)\times(n-1)$ submatrix of $M$,
and the corresponding submatrix $B^t$ of $M^t$. Then every value of
$t$ with $\cork(M^t)>1$ is a root of the algebraic equation
$\det(B^t)=0$, so only these have to be inspected. The first such
point will give a matrix $M^t\in\WW'_u$ with $\cork(M^t)>1$.

\begin{alg}[Double node]\label{ALG:2NODE} ~

{\it Input:} a vector $u\in\R^V$, two nodes $i$ and $j$ with
$u_i=u_j=0$, and a matrix $M\in\WW_u$.

{\it Output:} a matrix $M'\in\WW_u$ with at least two negative
eigenvalues.
\end{alg}

\smallskip

Subtract $t>0$ from both diagonal entries $M_{ii}$ and $M_{jj}$, to
get a matrix $M'$. Trivially $M'\in\WW_u$. Furthermore, if
$t>2\max\{|M_{ii}|,|M_{jj}|,|M_{ij}|\}$, then the submatrix of $M'$
formed by rows and columns $i$ and $j$ has negative trace and
positive determinant, and so it has two negative eigenvalues. This
implies that $M'$ has at least two negative eigenvalues.

\begin{alg}[Double cover]\label{ALG:2EDGE} ~

{\it Input:} a vector $u\in\R^V$, two edges $ab$ and $cd$ with
$u_a,u_c<0$ and $u_b,u_d>0$, and a matrix $M\in\WW_u$.

{\it Output:} a matrix $M'\in\WW_u$ with at least two negative
eigenvalues.
\end{alg}

Assume that $b\not=d$ (the case when $a\not=c$ can be treated
similarly). Define the symmetric matrix $N^{ab}\in\R^{V\times V}$ by
\[
(N^{ab})_{ij}=
  \begin{cases}
    u_b/u_a, & \text{if $\{i,j\}=\{a,b\}$},\\
    -u_b^2/u_a^2, & \text{if $i=j=a$}, \\
    -1, & \text{if $i=j=b$}, \\
    0, & \text{otherwise},
  \end{cases}
\]
and define $N^{cd}$ analogously. Then $N^{ab}u=N^{cd}u=0$, and so
$M'=M+tN^{ab}+tN^{cd}\in\WW_u$ for every $t>0$. Furthermore, if
$t>2\max\{|M_{bb}|,|M_{dd}|,|M_{bd}|\}$, then $M'$ has at least two
negative eigenvalues by the same argument as in Algorithm
\ref{ALG:2NODE}.

\subsection{Embedding in the line}\label{SEC:ALG1}

Now we come to the main algorithm for dimension 1.

\begin{alg}\label{ALG:MAIN} ~

{\it Input:} A connected graph $G=(V,E)$.

{\it Output:} Either an embedding $u:~V\to\R$ of $G$ (then $G$ is a
path), or a well-signed $G$-matrix with one negative eigenvalue and
corank at least $2$.
\end{alg}

{\bf Preparation.} We find a matrix $M\in\WW'(G)$. This is easy by
creating any well-signed $G$-matrix and subtracting its second
smallest eigenvalue from the diagonal. We may assume that
$\cork(M)=1$, else we are done.

Let $u\not=0$ be a vector in the nullspace of $M$, and let $\pi$ be
an eigenvector belonging to its negative eigenvalue. We apply
node-scaling, and get that the matrix $M'=\diag(\pi)M\diag(\pi)$ is
in $\WW'(G)$ and the vector $w=(u_i/\pi_i:~i\in V)$ is in its
nullspace. By Lemma \ref{LEM:ORDER}, this means that if we replace
$M$ by $M'$ and $u$ by $w$, then we get a vector $u\in\R^n$ and a
matrix $M\in\WW'_u$ such that every node $i$ with $u_i>0$ has a
neighbor $j$ with $u_j<u_i$, and every node $i$ with $u_i<0$ has a
neighbor $j$ with $u_j>u_i$.

Let us define a {\it cell} as an open interval between two
consecutive points $u_i$. If every cell is covered by only one edge,
then $G$ is a path and $u$ defines an embedding of $G$ in the line,
and we are done. Else, let us find a cell $(a,b)$ covered by at least
two edges that is nearest the origin. Replacing $u$ by $-u$ if
necessary, we may assume that $b>0$.

\medskip

{\bf Main step.} Below, we are going to maintain the following
conditions. We have a vector $u\in\R^V$ and a matrix $M\in\WW'_u$;
every node $i$ with $u_i>0$ has a neighbor $j$ with $u_j<u_i$; there
is a cell $(a,b)$ with $b>0$ that is doubly covered, and that is
nearest the origin among such cells.

We have to distinguish some cases.

{\bf Case 1.} If $a<0$, then we use the Double Cover Algorithm
\ref{ALG:2EDGE} to find a matrix $M'\in\WW_u$ with two negative
eigenvalues, and the Interpolation Algorithm \ref{ALG:INTER} returns
a matrix with the desired properties.

\medskip

{\bf Case 2.} If $a\ge 0$, then let $u_p$ be the smallest nonnegative
entry of $u$.

\smallskip

{\bf Case 2.1.} Assume that $u_p=0$. If there is a node $j\not=p$
with $u_j=0$, then run the Double Node algorithm \ref{ALG:2NODE} to
get a matrix in $\WW_u$ with at least two negative eigenvalues, and
we can finish by the Interpolation Algorithm \ref{ALG:INTER} again.
So we may assume that $u_j\not=0$ for $j\not=p$.

Let $(0,c)$ be the cell incident with $0$ $(c>0)$, and let $M'$ be
obtained from $M$ by replacing the $(p,p)$ diagonal entry by $0$,
then $M'\in\WW_u$. It follows by Lemma \ref{PROP:CONVHULL1} that $M'$
is not positive semidefinite. If $M'$ has more than one negative
eigenvalue, then we can run the Interpolation Algorithm
\ref{ALG:INTER}. So we may assume that $M'\in\WW'_u$.

For $t\in(0,c)$, consider the $G$-matrices $A^t$ defined for edges
$ij$ by
\[
A^t_{ij} = A^t_{ji} =
  \begin{cases}
    M_{ij},& \text{if $i,j\not=p$,}\\[6pt]
    \displaystyle\frac{u_j}{u_j-t}M_{pj}, & \text{if $i=p$,}\\
  \end{cases}
\]
and on the diagonal by
\[
A^t_{ii} = -\frac1{u_i-t}\sum_{j\in N(i)} A^t_{ij} (u_j-t).
\]
Clearly, $A^t$ is a well-signed $G$-matrix and $A^t(u-t)=0$. This
means that $A^t\in\WW_{u-t}$. Lemma \ref{PROP:CONVHULL1} implies that
$A^t$ has at least one negative eigenvalue. Furthermore, if $t\to 0$,
then $A^t_{ij}\to M_{ij}$; this is trivial except for $i=j=p$, when,
using that $\sum_{j\in N(p)} M_{pj} u_j = - M_{pp}u_p =0$, we have
\[
A^t_{pp} = \frac1{t}\sum_{j\in N(p)} M_{pj} u_j =0.
\]
Thus $A^t\to M'$ as $t\to 0$.

If the matrix $A^{c/2}$ has one negative eigenvalue, then replace $M$
by $A^{c/2}$ and $u$ by $u-c/2$, and return to the Main Step. Note
that the number of nodes with $u_i\ge 0$ has decreased, while those
with $u_i>0$ did not change.

If it has more than one, then consider the points $t\in(0,c/2]$ where
$\cork(A^t)>1$ (such a value of $t$ exists by Lemma \ref{SHIFT-NEW}).
These values of $t$ can be found like in the Interpolation Algorithm
\ref{ALG:INTER}. The smallest such value of $t$ gives $A^t\in\WW'(G)$
and $\cork(A^t)>1$, and we are done.

\smallskip

{\bf Case 2.2.} Assume that $u_p>0$. Let $\sigma$ and $\tau$ denote
the cells to the left and to the right of $u_p$ (so $0\in \sigma$).
There is no other node $q$ with $u_q=u_p$ (since from both nodes, an
edge would start to the left, whereas $0$ is covered only once). From
$u_p$, there is an edge starting to the left, and also one to the
right (since by connectivity, there is an edge covering $\tau$, and
this must start at $p$, since $\sigma$ is covered only once).
Therefore, $\WW_{u-u_p}\not=\emptyset$ by Lemma \ref{LEM:MU}.
Following the proof of this Lemma, we can construct a matrix
$B\in\WW_{u-u_p}$.

For $t\in[0,u_p)$, consider the $G$-matrices $B^t$ defined for edges
$ij$ by
\[
B^t_{ij} = B^t_{ji} =
  \begin{cases}
    B_{ij},& \text{if $i,j\not=p$,}\\
    \displaystyle\frac{u_j-u_p}{u_j-t}B_{pj}, & \text{if $i=p$,}\\
  \end{cases}
\]
and on the diagonal by
\[
B^t_{ii} = -\frac1{u_i-t}\sum_{j\in N(i)} B^t_{ij} (u_j-t).
\]
Clearly, $B^t$ is a well-signed $G$-matrix and $B^t(u-t)=0$.
Furthermore, $B^t\to B$ if $t\to u_p$.

If $B$ has one negative eigenvalue, then replace $M$ by $B$ and $u$
by $u-u_p$, and go to the Main Step. Note that the number of nodes
with $u_i>0$ has decreased, while those with $u_i\ge0$ did not
change.

If $B^0$ has more than one negative eigenvalue, then we call the
Interpolation Algorithm \ref{ALG:INTER}, to get a matrix in $\WW'_u$
with corank at least $2$. Finally, if $B$ has more than one negative
eigenvalue and $B^0$ has only one, then there must be values of $t$
such that $\cork(B^t)>1$. We can find these values just as in the
Interpolation Algorithm \ref{ALG:INTER}. For the smallest such value
of $t$ we have $B^t\in\WW'(G)$ and $\cork(B)>1$, and we are done.

\section{2-dimensional nullspace representations}\label{27ok15d}

\subsection{$G$-matrices and circulations}

Our goal in this section is to provide a characterization of
$G$-matrices and their nullspace representations in dimension $2$.

A {\it circulation} on an undirected simple graph $G$ is a real
function $f:~V\times V$ such that is supported on adjacent pairs, is
skew symmetric and satisfies the flow conditions:
\begin{equation*}
f(i,j)=0\ (ij\notin E),\quad f(i,j)=-f(j,i)\ (ij\in E),
\quad \sum_j f(i,j)=0\ (i\in V).
\end{equation*}
If we fix an orientation of the graph, then it suffices to specify
the values of $f$ on the oriented edges; the values on the reversed
edges follow by skew symmetry. A {\em positive circulation} on an
oriented graph $(V,A)$ is a circulation on the underlying undirected
graph that takes positive values on  the arcs in $A$.

For any representation $u:~V\to\R^k$, we define its {\it area-matrix}
as the (skew-symmetric) matrix $T=T(u)$ by $T_{ij}:=\det(u_i,u_j)$.
This number is the signed area of the parallelogram spanned by $u_i$
and $u_j$, and it can also be described as $T_{ij}=u_i\T u_j'$, where
$u'_j$ is the vector obtained by rotating $u_j$ counterclockwise over
$90^{\circ}$.

Given a graph $G$ and a representation $u:~V\to\R^2$ by nonzero
vectors, we define a directed graph $(V,A_u)$ and an undirected graph
$(V,E_u)$ by
\begin{align*}
A_u:&=\{(i,j)\in V\times V\mid ij\in E, T(u)_{ij}>0\}\\
E_u:&=\{ij\in E\mid T(u)_{ij}=0\}.
\end{align*}
So $E$ is partitioned into $A_u$ and $E_u$, where $(V,A_u)$ is an
oriented graph in which each edge is oriented counterclockwise as
seen from the origin. The graph $(V,E_u)$ consists of edges that are
contained in a line through the origin.

Given a representation $u:~V\to\R^2$, a circulation $f$ on $(V,A_u)$
and a function $g:~E_u\to\R$, we define a $G$-matrix $M(u,f,g)$ by
\[
M(u,f,g)_{ij} =
  \begin{cases}
    -f_{ij}/T_{ij}, & \text{if $ij\in A_u$}, \\
    g(ij), & \text{if $ij\in E_u$}.
  \end{cases}
\]
We define the diagonal entries by \eqref{8ok15d}, and let the other
entries be $0$.

\begin{lemma}\label{LEM:DECOMP2}
Let $G=(V,E)$ be a graph, let $u:~V\to\R^2$ be a representation of
$V$ by nonzero vectors. Then
\[
\MM_u = \bigl\{M(u,f,g):~\text{\rm $f$ is a circulation on $(V,A_u)$ and
$g:~E_u\to\R$}\bigr\}.
\]
\end{lemma}

\begin{proof}
First, we prove that $M(u,f,g)\in\MM_u$ for every circulation on
$(V,A_u)$ and every $g:~E_u\to\R$. Using that
$M(u,f,g)=M(u,f,0)+M(u,0,g)$, it suffices to prove that
$M(u,f,g)\in\MM_u$ if either $g=0$ or $h=0$. If $M=M(u,f,0)$, then
using that $f$ is a circulation, we have
\[
\Bigl(\sum_j M_{ij}u_j\Bigr)\T u_i' = \sum_j f_{ij} = 0.
\]
This means that $\sum_jM_{ij}u_j\T$ is orthogonal to $u_i'$, and so
parallel to $u_i$. As remarked above, this means that
$M(u,f,0)\in\MM_u$. If $M=M(u,0,g)$, then for every $i\in V$,
\[
\sum_{j\in N(i)} M_{ij}u_j = \sum_{j:\,ij\in E_u} g(ij) u_j
\]
This vector is clearly parallel to $u_i$, proving that
$M(u,0,g)\in\MM_u$.

Second, given a matrix $M\in\MM_u$, define $f_{ij}=-T_{ij}M_{ij}$ for
$ij\in A_u$ and $g_{ij}=M_{ij}$ for $ij\in E_u$. Then $f$ is a
circulation. Indeed, for $i\in V$,
\[
\sum_{ij\in A_u} f_{ij} = -\sum_{ij\in A_u} M_{ij}u_j\T u_i'
= -\sum_{j\in V} M_{ij}u_j\T u_i' = \Bigl(-\sum_{j\in V} M_{ij}u_j\Bigr)\T u_i' = 0.
\]
Furthermore, $M(u,f,g)=M$ by simple computation.
\end{proof}

Note that the $G$-matrix $M(u,f,g)$ is well-signed if and only if $f$
is a positive circulation on $(V,A_u)$ and $g<0$. Thus,

\begin{corollary}\label{COR:DECOMP2}
Let $G=(V,E)$ be a graph, let $u:~V\to\R^2$ be a representation of
$V$ by nonzero vectors. Then
\begin{align*}
\WW_u = \bigl\{M(u,f,g):~&\text{\rm $f$ is a positive circulation on}~(V,A_u),~\\
&g:~E_u\to\R,\ g<0\bigr\}.
\end{align*}
\end{corollary}

In particular, it follows that $\WW_u\neq\emptyset$ if and only if
$A_u$ carries a positive circulation. This happens if and only if
each arc in $A_u$ is contained in a directed cycle in $A_u$; that is,
if and only if each component of the directed graph $(V,A_u)$ is
strongly connected.

The signature of eigenvalues of $M(u,f,g)$ is a more difficult
question, but we can say something about $M(u,0,g)$ if $g<0$. Let $H$
be a connected component of the graph $(V,E_u)$, and let $M_H$ be the
submatrix of $M(u,0,g)$ formed by the rows and columns whose index
belongs to $V(H)$. Then $M_H$ is a well-signed $H$-matrix. The
vectors $u_i$ representing nodes $i\in V(H)$ are contained in a
single line through the origin. Lemma \ref{PROP:CONVHULL1} implies
that $M_H$ has at least one negative eigenvalue unless $u(V(H))$ is
contained in a semiline starting at the origin. Let us call such a
component {\it degenerate}. Then we can state:

\begin{lemma}\label{LEM:MUHG}
Let $u:~V\to\R^2$ be a representation of $V$ with nonzero vectors,
and let $g:~E_u\to\R$ be a function with negative values. Then the
number of negative eigenvalues of $M(u,0,g)$ is at least the number
of non-degenerate components of $(V,E_u)$.
\end{lemma}

\subsection{Shifting the origin}

For a representation $(u_1,\dots, u_n)$ in $\R^2$ and $p\in\R^2$, let
us write $u-p$ for the representation $(u_1-p,\dots,u_n-p)$.

Consider the cell complex made by the (two-way infinite) lines
through distinct points $u_i$ and $u_j$ with $ij\in E$. The 1- and
2-dimensional cells are called {\em $1$-cells} and {\em $2$-cells},
respectively. Two cells $c$ and $d$ are {\em incident} if
$d\subseteq\overline c\setminus c$ or
$c\subseteq\overline{d}\setminus{d}$.

Two points $p$ and $q$ belong to the same cell if and only if
$A_{u-p}=A_{u-q}$ and $E_{u-p}=E_{u-q}$. Hence, for any cell $c$, we
can write $A_c$ and $E_c$ for $A_{u-p}$ and $E_{u-p}$, where $p$ is
an arbitrary element of $c$. For any cell $c$, set
$\WW_c:=\bigcup_{p\in c}\WW_{u-p}$. It follows by Lemma
\ref{LEM:DECOMP2} that if $\WW_c\not=\emptyset$, then
$\WW_{u-p}\not=\emptyset$ for every $p\in c$. It also follows that
$\WW_c$ is connected for each cell $c$, as it is the range of the
continuous function $M(u-p,f,g)$ on the connected topological space
of triples $(p,f,g)$ where $p\in c$, $f$ is a positive circulation on
$A_c$, and $g$ is a negative function on $E_c$.

The following lemma is an essential tool in the proof.

\begin{lemma}\label{13ok15a}
Let $c$ be a cell with $\WW_c\neq\emptyset$ and let
$q\in\overline{c}$. Then $M(u-q,0,g)\in\overline{\WW}_c$ for some
negative function $g$ on $E_{u-q}$.
\end{lemma}

\begin{proof}
Choose any $p\in c$. Note that $q\in\overline c$ implies that
$E_{u-p}\subseteq E_{u-q}$. Let $M\in\WW_{u-p}$, then by Lemma
\ref{LEM:DECOMP2} we can write $M=M(u-p,h,g')$ with some positive
circulation $h$ on $A_{u-p}$ and negative function $g'$ on $E_{u-p}$.
Define $g(ij)=M_{ij}$ for $ij\in E_{u-q}$. For $\alpha\in(0,1]$,
define $p_\alpha=(1-\alpha)q+\alpha p$, and consider the $G$-matrices
$M_\alpha=M(u-p_\alpha,\,\alpha h,\,g')$. Clearly $M_\alpha\in\WW_c$.
It suffices to prove that
\begin{equation}\label{13ok15b}
M_\alpha\to M(u-q,0,g)\qquad (\alpha\to 0).
\end{equation}

Consider any position $(i,j)$ with $i\neq j$. If $ij\in E_{u-p}$,
then the $(i,j)$ matrix entries in $M_\alpha$ and $M(u-q,0,g)$ are
both equal to $g'(ij)$, independently of $\alpha$. If $ij\not\in
E_{u-p}$, then for each $\alpha\in(0,1]$ we have $ij\not\in
E_{u-\alpha p}$, and
\begin{equation}\label{7ok15c}
(M_\alpha)_{ij} =\frac{-\alpha h_{ij}}{T(u-p_{\alpha})_{ij}}.
\end{equation}
If $ij\in E_{u-q}\setminus E_{u-p}$, then there is a line through
$u_i$, $u_j$, and $q$. Hence $T(u-p_{\alpha})_{ij}=\alpha
T(u-p)_{ij}$ for each $\alpha\in(0,1]$, and so
\[
(M_\alpha)_{ij} =\frac{-h_{ij}}{T(u-p)_{ij}} =M_{ij}.
\]
If $ij\not\in E_{u-q}$, then (\ref{7ok15c}) implies that
$(M_\alpha)_{ij}\to 0$ as $\alpha\to0$, since $\lim_{\alpha\to
0}T(u-p_{\alpha})_{ij}=T(u-q)_{ij}\not=1$.

So \eqref{13ok15b} holds on all off-diagonal positions. By
\eqref{8ok15d}, it holds for the diagonal entries as well.
\end{proof}

\begin{corollary}\label{13ok15all}
Let $c$ be a cell with $\WW_c\neq\emptyset$ and $q\in\overline{c}$.
Then for every matrix $M\in\WW_{u-q}$ there is a matrix
$M'\in\WW_{u-q}\cap\overline{\WW}_c$ that differs from $M$ only on
entries corresponding to edges in $E_{u-q}$ and on the diagonal
entries.
\end{corollary}

\begin{proof}
By Lemma \ref{LEM:DECOMP2} we can write $M=M(u-q,f,g)$ with some
positive circulation $f$ on $A_{u-q}$ and negative function $g$ on
$E_{u-q}$. By Lemma \ref{13ok15a}, there is a negative function $g'$
on $E_{u-q}$ such that $M(u-q,0,g')\in\overline{\WW}_c$. There are
points $p_k\in c$ and matrices $M_k\in\WW_{u-p_k}$ such that $M_k\to
M(u-q,0,g')$ as $k\to\infty$. Then $M_k+M(u-p_k,f,0)$ belongs to
$\WW_{u-p_k}$ and $M_k+M(u-p_k,f,0)\to M(u-q,0,g') + M(u-q,f,0) =
M(u-q,f,g')$ as $k\to\infty$, showing that $M'=M(u-q,f,g')$ belongs
to $\overline{\WW}_c$. Furthermore, $M-M'=M(u-q,0,g-g')$ is nonzero
on entries in $E_{u-q}$ and on the diagonal entries only.
\end{proof}

\begin{corollary}\label{13ok15bx}
If $c$ and $d$ are incident cells, then $\WW_c\cup\WW_{d}$ is
connected.
\end{corollary}

\begin{proof}
We may assume that $d\subseteq \overline c\setminus c$, and that both
$\WW_c$ and $\WW_{d}$ are nonempty (otherwise the assertion follows
from the connectivity of $\WW_c$ and $\WW_d$).

Choose $q\in d$. Since $\WW_{d}\neq\emptyset$, Corollary
\ref{13ok15all} implies that $\WW_{d}$ and $\overline{\WW}_c$
intersect, and by the connectivity of $\WW_c$ and $\WW_d$, this
implies that $\WW_c\cup\WW_{d}$ is connected.
\end{proof}

Call a segment $\sigma$ in the plane {\em separating}, if $\sigma$
connects points $u_a$ and $u_b$ for some $a,b\in V$, with the
property that $V\setminus\{a,b\}$ can be partitioned into two
nonempty sets $X$ and $Y$ such that no edge of $G$ connects $X$ and
$Y$ and such that the sets $\{u_i\mid i\in X\}$ and $\{u_i\mid i\in
Y\}$ are on distinct sides of the line through $\sigma$. Note that
this implies that $\sigma$ is a $1$-cell.

\begin{lemma}\label{8ok15a}
Let $G$ be a connected graph, and let $\sigma$ be a separating
segment connecting $u_i$ and $u_j$, with incident $2$-cells $R$ and
$Q$. If $\WW_{\sigma}\cup\WW_R\neq\emptyset$, then $A_Q$ contains a
directed circuit traversing $ij$.
\end{lemma}

\begin{proof}
We may assume that $\sigma$ connects $u_1$ and $u_2$, and that edge
$12$ of $G$ is oriented from $1$ to $2$ in $A_Q$. Let $\ell$ be the
line through $\sigma$, and let $H$ and $H'$ be the open halfplanes
with boundary $\ell$ containing $Q$ and $R$, respectively.

Choose $p\in \sigma\cup R$ with $\WW_{u-p}\neq\emptyset$. Note that
$A_Q$ and $A_{u-p}$ differ only for edge $12$. Any edge $ij\neq 12$
has the same orientation in $A_Q$ as in $A_{u-p}$.

Since $H$ contains points $u_i$, since $G$ is connected, and since
$\ell$ crosses no $u_iu_j$ with $ij\in E$, $G$ has an edge $1k$ or
$2k$ with $u_k\in H$. By symmetry, we can assume that $2k$ is an
edge. Then in $A_{u-p}$, edge $2k$ is oriented from $2$ to $k$. As
$\WW_{u-p}\neq\emptyset$, $A_{u-p}$ has a positive circulation. So
$A_{u-p}$ contains a directed circuit $D$ containing $2k$. The edge
preceding $2k$, say $j2$, must have $u_j\in H'$, as $p$ belongs to
$\sigma\cup R$. Therefore, since $\{1,2\}$ separates nodes $k$ and
$j$, $D$ traverses node $1$. So the directed path in $D$ from $2$ to
$1$ together with the edge $12$ forms the required directed circuit
$C$ in $A_{u-q}$.
\end{proof}

\begin{corollary}\label{14ok15b}
Let $G$ be a connected graph, let $\sigma$ be a separating segment,
and let $R$ be a $2$-cell incident with $\sigma$. Then
$\WW_{\sigma}\neq\emptyset$ if and only if $\WW_R\neq\emptyset$.
\end{corollary}

\begin{proof}
Let $\sigma$ connect $u_1$ and $u_2$. If $\WW_{\sigma}\neq\emptyset$,
then $A_{\sigma}$ has a positive circulation $f'$. By Lemma
\ref{8ok15a}, $A_R$ contains a directed circuit $C$ traversing $12$.
Let $f$ be the incidence vector of $C$. Then $f'+f$ is a positive
circulation on $A_R$. So $\WW_R\neq\emptyset$.

Conversely, if $\WW_R\not=\emptyset$, then $A_R$ has a positive
circulation $f$. By Lemma \ref{8ok15a}, $A_R$ contains a directed
cycle through the arc $21$, which gives a directed path $P$ from $1$
to $2$ not using $12$. It follows that by rerouting $f(1,2)$ over
$P$, we obtain a positive circulation on $A_{\sigma}$, showing that
$\WW_{\sigma}\neq\emptyset$.
\end{proof}

\subsection{Outerplanar nullspace embeddings}

Let $G=(V,E)$ be a graph. A mapping $u:~V\to\R^2$ is called {\em
outerplanar} if its extension to the edges gives an embedding of $G$
in the plane, and each $u_i$ is incident with the unbounded face of
this embedding.

\begin{theorem}\label{13ok15d}
Let $G$ be a $2$-connected graph with $\kappa(G)=2$. Then the
normalized nullspace representation defined by any well-signed
$G$-matrix with one negative eigenvalue and with corank $2$ is an
outerplanar embedding of $G$.
\end{theorem}

\smallskip
\noindent {\bf Proof.} Let $u$ be such a normalized nullspace
representation (this exists by Corollary \ref{1NODE}). Let $K$ be the
convex hull of $u(V)$. Since all $u_i$ have unit length, each $u_i$
is a vertex of $K$. We define a {\it diagonal} as the line segment
connecting points $u_i\not=u_j$, where $ij\in E$. We don't know at
this point that the points $u_i$ are different and that diagonals do
not cross; so the same diagonal may represent several edges of $G$,
and may consist of several $1$-cells.

Let $P$ denote the set of points $p\in\R^2\setminus u(V)$ with
$\WW'_{u-p}\neq\emptyset$. Clearly, the origin belongs to $P$. Lemma
\ref{PROP:CONVHULL1}(b) implies that

\begin{claim}\label{CLAIM:0}
$P$ is contained in the interior of $K$.
\end{claim}

\noindent(It will follow below that $P$ is equal to the interior of
$K$.)

Consider again the cell complex into which the diagonals cut $K$. By
the connectivity of the sets $\WW_c$ and by Lemma \ref{SHIFT-NEW},
$P$ is a union of cells.

\begin{claim}\label{CLAIM:1}
$\overline{P}$ cannot contain a point $u_i=u_j$ for two distinct
nodes $i$ and $j$.
\end{claim}

Indeed, since $u_i=u_j$ is a vertex of the convex hull of $u(V)$, we
can choose $p\in P$ close enough to $v$ so that it is not in the
convex hull of $u(V)\setminus \{v\}$. This, however, contradicts
Lemma \ref{27ok15e}.

\begin{claim}\label{CLAIM:2}
No point $p\in \overline{P}\setminus u(V)$ is contained in two
different diagonals.
\end{claim}

Indeed, consider any cell $c\subseteq P$ with $p\in\overline c$.
Since $\WW_c\neq\emptyset$, Lemma \ref{13ok15a} implies that there is
a negative function $g$ on $E_{u-p}$ such that
$M(u-p,0,g)\in\overline{\WW}_c$. As all matrices in $\WW_c$ have
exactly one negative eigenvalue, $M(u-p,0,g)$ has at most one
negative eigenvalue. Lemma \ref{LEM:MUHG} implies that $(V,E_{u-p})$
has at most one non-degenerate component. But every diagonal
containing $p$ is contained in a non-degenerate component of
$(V,E_{u-p})$, and these components are different for different
diagonals, so $p$ can be contained in at most one diagonal. This
proves Claim \ref{CLAIM:2}.

It is easy to complete the proof now. Clearly, $P$ is bounded by one
or more polygons. Let $p$ be a vertex of $\overline{P}$, and assume
that $p\notin u(V)$. Then $p$ belongs to two diagonals (defining the
edges of $P$ incident with $p$), contradicting Claim \ref{CLAIM:2}.
Thus all vertices of $P$ are contained in $u(V)$. This implies that
$\overline{P}$ is a convex polygon spanned by an appropriate subset
of $u(V)$.

To show that $\overline{P}=K$, assume that the boundary of $P$ has an
edge $\sigma$ contained in the interior of $K$ and let $R\subseteq P$
be a $2$-cell incident with $\sigma$, and let $Q$ be the $2$-cell
incident with $\sigma$ on the other side. Clearly,
$\WW_R\neq\emptyset$, and by Corollary \ref{14ok15b},
$\WW_{\sigma}\neq\emptyset$ and by the same Corollary,
$\WW_Q\neq\emptyset$. The sets $\WW_{\sigma}\cup\WW_R$ and
$\WW_{\sigma}\cup\WW_Q$ are connected by Corollary \ref{13ok15bx},
and hence so is $\WW_{\sigma}\cup\WW_R\cup\WW_Q$. We also know that
$\WW'\cap\WW_R\neq\emptyset$. Since $\WW'$ is open and closed in
$\WW$ (Lemma \ref{SHIFT-NEW}, note that in this case
$\WW'=\WW'\cap\WW_0$ as $\kappa(G)=2$), we conclude that
$\WW'\cap\WW_Q\neq\emptyset$, i.e., $Q\subseteq P$. But this
contradicts the definition of $\sigma$.

Thus $P$ is equal to the interior of $K$. Claim \ref{CLAIM:1} implies
that the points $u_i$ are all different, and Claim \ref{CLAIM:2}
implies that the diagonals do not cross.\bx

\subsection{Algorithm}

The considerations in this section give rise to a polynomial
algorithm achieving the following.

\begin{alg}\label{ALG:MAIN2} ~

{\it Input:} A 2-connected graph $G=(V,E)$.

{\it Output:} Either an outerplanar embedding $u:~V\to\R^2$ of $G$,
or a well-signed $G$-matrix with one negative eigenvalue and corank
at least $3$.
\end{alg}

The algorithm progresses along the same lines as the algorithm in
Section \ref{SEC:ALG1}, with auxiliary algorithms analogous to those
in Section \ref{SEC:ALG-AUX}. We omit the details.

\begin{remark}\label{REM:POLYTOPE}
Suppose that the input to our algorithm is a $3$-connected planar
graph. Then the algorithm outputs a well-signed $G$-matrix with one
negative eigenvalue and corank at least $3$. Computing the nullspace
representation defined by this matrix, and performing node-scaling as
described in \cite{L}, we get a representation of $G$ as the skeleton
of a $3$-polytope.
\end{remark}

\bigskip \noindent {\em Acknowledgement.} We thank Bart Sevenster for
helpful discussion on $\kappa(G)$.

\end{document}